\documentclass[12pt]{amsart}
\usepackage{amsmath, amscd, amssymb, amsthm, verbatim}
\allowdisplaybreaks

\newtheorem{theorem}{Theorem}[section]

\newtheorem{lemma}[theorem]{Lemma}
\newtheorem{proposition}[theorem]{Proposition}
\newtheorem{remark}[theorem]{Remark}

\newtheorem{definition}[theorem]{Definition}

\DeclareMathOperator{\Rc}{Rc}
\DeclareMathOperator{\Div}{Div}

\def\ppt{\frac{\partial}{\partial t}}

\def\RR{{\mathrm R}}

\def\Rc{{\mathrm {Rc}}}

\def\SS{{\mathrm S}}

\newcommand{\D}{\mathcal{D}}

\begin{document}
\title[Harnack estimates on evolving manifolds]{Harnack estimates for conjugate heat kernel on evolving manifolds}

\author{Xiaodong Cao}
\address{Department of Mathematics,
 Cornell University, Ithaca, NY 14853}
\email{cao@math.cornell.edu}

\author{Hongxin Guo}

\address{School of mathematics and information science, 
Wenzhou University, Wenzhou, Zhejiang 325035, China.
}
\email{guo@wzu.edu.cn}

\author{Hung Tran}

\address{Department of Mathematics,
 University of California, Irvine, CA 92697}
\email{hungtt1@math.uci.edu}




\date{\today}

\begin{abstract}
In this article we derive Harnack estimates for conjugate heat kernel in an abstract geometric flow.
Our calculation involves a correction term $\D$. When $\D$ is nonnegative, we are able to obtain a
Harnack inequality. Our abstract formulation provides a unified framework for some known results,
in particular including corresponding results of Ni \cite{ni04entropy},
Perelman \cite{perelman1} and Tran \cite{Tran1} as special cases. Moreover it leads to new results
in the setting of Ricci-Harmonic flow
and mean curvature flow in Lorentzian manifolds with nonnegative sectional curvature.
\end{abstract}

\maketitle
\tableofcontents
\section{Introduction}
Assume that $M$ is an $n$-dimensional closed manifold endowed with a one-parameter family of Riemannian
metrics $g(t)$, $t \in [0, T]$, evolving by 
\begin{equation}\label{flow equation}
\frac{\partial g(t,x)}{\partial t}=-2\alpha(t,x).
\end{equation}
Here $\alpha(t,x)$ is a one-parameter family of smooth symmetric 2-tensors on $M$.
In particular, when $\alpha=\Rc$, Eq.(\ref{flow equation}) is R. Hamilton's Ricci flow. 
Let
$$
\SS(t,x)\doteqdot g^{ij}\alpha_{ij}
$$
be the trace of $\alpha$ with respect to the time-dependent metric $g(t)$. \\

In \cite{Muller10}, R. M\"uller studied reduced volumes for the abstract flow (\ref{flow equation}) and defined the following quantity for tensor $\alpha$ and  vector $V$,
\begin{align}\label{D definition}
\mathcal D_\alpha(V)\doteqdot & \frac{\partial \SS}{\partial t}-\Delta \SS-2|\alpha|^2\nonumber\\
& +2\left(\Rc-\alpha\right)(V,V)+\langle4\operatorname{Div}(\alpha)-2\nabla \SS, V\rangle,
\end{align}
where $\Div$ is the divergence operator defined by $\Div(\alpha)_k=g^{ij}\nabla_i\alpha_{jk}$ (in local coordinates).
Under the assumption that $\D_\alpha \geq 0$, M\"uller obtained monotonicity of the reduced volumes \cite{Muller10}.
Most recently, in \cite{GPT}, the authors proved monotonicity for the entropy and the lowest eigenvalue.  In \cite{GH},
a Harnack inequality for positive solutions of the conjugate heat equation and heat equation with potential has been proved.\\

The main purpose of this article is to derive Harnack inequalities for a conjugate heat kernel in the abstract setting with $\D_\alpha\ge 0.$

\subsection{Main Results}

We consider $(M, g(t))$, $0\leq t\leq T$, to be a solution of (\ref{flow equation}) and $\tau\doteqdot T-t$, $$\Box^{\ast}\doteqdot -\frac{\partial}{\partial t}-\Delta+\SS=\frac{\partial}{\partial \tau}-\Delta+\SS.$$ A function $u=(4\pi\tau)^{-n/2}e^{-f}$ is a solution to the conjugate heat equation if,
\begin{align}
\label{conjugate heat eqn}
\Box^{\ast} u=0.
\end{align}
We also denote $$H(x,t;y,T)=(4\pi(T-t))^{-n/2}e^{-h}=(4\pi\tau)^{-n/2}e^{-h}$$ to be a heat kernel. That is, based at a fixed $(x,t)$, H is the fundamental solution of heat equation $\Box{H}=0$, and similarly for fixed $(y,T)$ and conjugate heat equation $\Box^{\ast}H=0$.
Our first result is computational.

\begin{theorem}
\label{conjHar}
Let $$v= \left(\tau\left(2\Delta f-|\nabla f|^2+\SS\right)+f-n\right)u,$$ then we have
\begin{align}
\Box^{\ast} v &=-2\tau u\left|\alpha+\nabla\nabla f-\frac g{2\tau}\right|^2-\tau u\mathcal D_\alpha(\nabla f).
\end{align}
\end{theorem}

Secondly, we obtain the following Harnack estimate.

\begin{theorem}
\label{mainth}
If $\D_{\alpha}\ge 0$, then the following inequality holds,
\begin{align}\label{main harnack}
\tau\left(2\Delta h-|\nabla h|^2+\SS\right)+h-n\le0.
\end{align}
\end{theorem}

\begin{remark}
For the Ricci flow, where $\alpha=\Rc$, one has $\D = 0$; \eqref{main harnack}
has been proved by G. Perelman \cite{perelman1}. On a static Riemannian manifold where $\alpha=0$ one has $\D=\Rc$, and \eqref{main harnack} has been proved by L. Ni in \cite{ni04entropy} for static manifolds with nonnegative Ricci curvature. Another special case of \eqref{main harnack} was recently proved by the third author for the extended Ricci flow in \cite{Tran1}. \eqref{main harnack} is new for M\"uller's Ricci-Harmonic flow and  mean curvature flow in a Lorentzian manifold with nonnegative sectional curvature. The detailed calculations of $\D$ can be found in \cite{Muller10}.
\end{remark}

\textbf{Acknowledgement}.  X. Cao was partially supported by a grant from the
Simons Foundation (\#280161) and by the Jeffrey Sean
Lehman Fund from Cornell University; H. Guo was supported by NSF of China (Grant No. 11171143)
  and Zhejiang Provincial Natural Science Foundation of China (Project No. LY13A010009).

\section{Preliminaries}
\subsection{Evolution Equations}
In this section, we collect several evolution equations and prove Theorem \ref{conjHar}.\\

For the Laplace-Beltrami operator $\Delta$ with respect to $g(t)$ we have,
\begin{align}\label{evolution of laplace}
\left(\frac{\partial}{\partial t}\Delta\right) f=2\langle\alpha, \nabla\nabla f\rangle+\langle 2\operatorname{Div}(\alpha)-\nabla \SS,\nabla f\rangle,
\end{align}
where $f$ is any smooth function on $M$. This formula can be found in standard textbooks, for instance \cite{chowluni}.\\

Now we assume $u$ is a solution to the conjugate heat equation.
The operator $-\Box^{\ast}$ acting on the term $u\log u$ produces,
\begin{align}\label{1st derivative of ulogu}
-\Box^{\ast} u\log u=u|\nabla\log u|^2+u\SS.
\end{align}
The same operator acts once more and we have,
\begin{align}
&-\Box^{\ast} \left(u|\nabla\log u|^2+u\SS\right) \nonumber\\
&=2u\alpha(\nabla\log u, \nabla\log u)+4\langle{\nabla \SS,\nabla u}\rangle+u\frac{\partial \SS}{\partial t}+2u|\nabla\nabla\log u|^2\nonumber\\
&+2u\Rc(\nabla\log u,\nabla \log u)+u\Delta \SS\nonumber\\
\label{2nd derivative of ulogu}
&=2u|\nabla\nabla\log u-\alpha|^2+4u\langle{\alpha,\nabla\nabla\log u}\rangle+2u\alpha(\nabla\log u,\nabla\log u)\\
&-2u|\alpha|^2+4\langle{\nabla \SS,\nabla u}\rangle+u\frac{\partial \SS}{\partial t}+2u\Rc(\nabla\log u,\nabla \log u)+u\Delta \SS. \nonumber
\end{align}

Notice that,
\begin{align}
-\Box^{\ast}(\Delta u)=
&2\langle{\alpha,\nabla\nabla u}\rangle+\langle2\Div(\alpha)-\nabla \SS, \nabla u\rangle+2\langle{\nabla \SS,\nabla u}\rangle+u\Delta \SS\nonumber\\
=&2u\langle{\alpha,\nabla\nabla\log u}\rangle+2u\alpha(\nabla\log u,\nabla \log u)\\
&+\langle2\Div(\alpha)-\nabla \SS, \nabla u\rangle+2\langle{\nabla \SS,\nabla u}\rangle+u\Delta \SS. \nonumber
\end{align}
Thus, by (\ref{D definition}),  we have, for $V=-\nabla\log u$,
\begin{equation}
-\Box^{\ast}\left(u|\nabla\log u|^2+u\SS-2\Delta u\right)
=2u|\nabla\nabla\log u-\alpha|^2+u\mathcal D_\alpha(V).
\end{equation}

Moreover,
\begin{align}
&-\Box^{\ast}\left(\tau \left(u|\nabla\log u|^2+u\SS-2\Delta u\right)-u\log u-\frac {nu} 2\log\tau\right)\\\nonumber
&=2\tau u|\alpha-\nabla\nabla\log u-\frac g{2\tau}|^2+\tau u\mathcal D_\alpha(-\nabla\log u).
\end{align}
In the calculation above, if we add a normalization term $c_nu$ to the left hand side, we get the same on the right hand side since
$\Box^{\ast}u=0.$
Thus, we have the following result.

\begin{lemma}
\begin{align}
&\Box^{\ast}\left[\tau \left(u|\nabla\log u|^2+u\SS-2\Delta u\right)
-u\log u-\frac {nu} 2\log\left(4\pi\tau\right)-nu\right]\\\nonumber
&=-2\tau u\left|\alpha-\nabla\nabla\log u-\frac g{2\tau}\right|^2-\tau u\mathcal D_\alpha(-\nabla\log u).
\end{align}
\end{lemma}

Theorem \ref{conjHar} follows by realizing that $f\doteqdot -\log u-\frac n 2\log(4\pi\tau).$

\subsection{Asymptotic Behavior and Reduced Geometry}

Let's recall the asymptotic behavior of the heat kernel as $t\rightarrow T$.
\begin{theorem}\cite[Theorem 24.21]{chowetc3} \label{asymptoticconj}
For $\tau=T-t$,
\begin{equation*}
H(x,t;y,T) \sim \frac{e^{-\frac{d_{T}^{2}(x,y)}{4\tau}}}{(4\pi\tau)^{n/2}}\Sigma_{j=0}^{\infty}\tau^{j}u_{j}(x,y,\tau).
\end{equation*}
More precisely, there exist $t_{0}>0$ and a sequence $u_{j}\in C^{\infty}(M\times M\times [0,t_{0}])$ such that,
\begin{equation*}
H(x,t;y,T)-\frac{e^{-\frac{d_{T}^{2}(x,y)}{4\tau}}}{(4\pi\tau)^{n/2}}\Sigma_{j=0}^{k}\tau^{j}u_{j}(x,y,T-l)=w_{k}(x,y,\tau),
\end{equation*}
with
\begin{align*}
u_{0}(x,x,0)&=1,\\
w_{k}(x,y,\tau)&=O(\tau^{k+1-\frac{n}{2}}),
\end{align*}
as $\tau\rightarrow 0$ uniformly for all $x,y\in M$.
\end{theorem}

Then following \cite{Muller10}, we can define reduced length and distance.

\begin{definition}
Given $\tau(t)=T-t$, we define the $\mathcal{L}$-length of a curve $\gamma: [\tau_{0},\tau_{1}]\mapsto N$, $[\tau_{0},\tau_{1}]\subset[0,T]$ by,
\begin{equation}
\mathcal{L}(\gamma):=\int_{\tau_{0}}^{\tau_{1}}\sqrt{\tau}(\SS(\gamma(\tau))+|\dot{\gamma}(\tau)|^2)d\tau.
\end{equation}
For a fixed point $y\in N$ and $\tau_{0}=0$, the backward reduced distance is defined as,
\begin{equation}
\label{adaptedRD}
\ell (x,\tau_{1}):=\inf_{\gamma\in \Gamma}\{\frac{1}{2\tau_{1}}\mathcal{L}(\gamma)\},
\end{equation}
where $\Gamma=\{\gamma:[0,\tau_{1}]\mapsto M, \gamma(0)=y, \gamma(\tau_{1})=x\}$.\\
The backward reduced volume is defined as 
\begin{equation}
V(\tau):=\int_{M}(4\pi\tau)^{-n/2}e^{-\ell (y,\tau)}d\mu_{\tau}(y).
\end{equation}
\end{definition}

The next result, mainly from \cite{Muller10}, relates the reduced distance defined in (\ref{adaptedRD}) with the distance at time T.
\begin{lemma}
\label{compareRDandconjheat}
Let $L(x,\tau)=4\tau\ell(x,\tau)$ then we have the followings:\\
{\bf a.} Assume that there exists $k_{1},k_{2}\geq 0$ such that $-k_{1}g(t)\leq \alpha(t) \leq k_{2}g(t)$ for $t\in[0,T]$, then $L$ is smooth amost everywhere and a local Lipschitz function on $N\times [0,T]$. Furthermore,
$$e^{-2k_{1}\tau}d_{T}^2(x,y)-\frac{4k_{1}n}{3}\tau^2\leq L(x,\tau)\leq e^{2k_{2}\tau}d_{T}^2(x,y)+\frac{4k_{2}n}{3}\tau^2 .$$
{\bf b.} If $\D_{\alpha}\ge 0$, then $\Box^{\ast}\Big(\frac{e^{-\frac{L(x,\tau)}{4\tau}}}{(4\pi\tau)^{n/2}}\Big)\leq 0.$\\
{\bf c.} For the same point y in the definition of reduced distance and $H(x,t;y,T)=(4\pi\tau)^{-n/2}e^{-h}$, then $h(x,t;y,T)\leq \ell (x,T-t)$.
\end{lemma}
\begin{proof} Parts {\bf a.} and {\bf b.} follow from \cite[Lemmas 4.1, 5.15]{Muller10} respectively.

For part {\bf c.} we provide a brief argument here (for more details, see \cite[Lemma 16.49]{chowetc2}).

We first observe that part {\bf a.} implies $\lim_{\tau\rightarrow 0}L(x,\tau) =d_{T}^2(y,x)$ and,
\[\lim_{\tau\rightarrow 0}\frac{e^{-\frac{L_{w}(x,\tau)}{4\tau}}}{(4\pi\tau)^{n/2}}=\delta_{y}(x),\]
since  Riemannian manifolds locally look like Euclidean. It then follows from part {\bf b.} and maximum principle that,
\begin{equation*}
H(x,t;y,T)\geq \frac{e^{-\frac{L(x,\tau)}{4\tau}}}{(4\pi\tau)^{n/2}}=\frac{e^{-\frac{L(x,T-t)}{4\tau}}}{(4\pi(T-t))^{n/2}}.
\end{equation*}
Hence, 
\begin{equation*}
h(x,t;y,T)\leq \frac{L(x,\tau)}{4\tau}=\ell (x,\tau)=\ell (x,T-l).
\end{equation*}
\end{proof}
\subsection{Entropy Formulas}
In this subsection, we define several functionals and collect their properties.  
\begin{definition}
Along flow (\ref{flow equation}), for $h$ satisfying 
$\int_{M}(4\pi\tau)^{-n/2}e^{-h}d\mu=1$, we define 
\begin{equation}
\label{defineW}
\mathbb{W}_{\alpha}(g,\tau,h)\doteqdot \int_{M}\Big(\tau(|\nabla h|^2+\SS)+(h-n)\Big)(4\pi\tau)^{-n/2}e^{-h}d\mu.
\end{equation}
Associated functionals are defined as follows:
\begin{align}
\mu_{\alpha}(g,\tau)&=\inf_{f}{\mathbb{W}_{\alpha}(g, h, \tau)},\\
\upsilon_{\alpha}(g)&=\inf_{\tau>0}{\mu_{\alpha}(g,\tau)}.
\end{align}
\end{definition}
\begin{remark} Since $\alpha$ is a $(2,0)$-tensor, $\SS$ scales like the inverse of the metric. Thus, these functionals satisfy diffeomorphism invariance and the following scaling rules:
\begin{align*}
\mathbb{W}_{\alpha}(g,\tau,h)&= \mathbb{W}_{\alpha}(cg,c\tau,h),\\
\mu_{\alpha}(g,\tau)&=\mu_{\alpha}(cg, c\tau),\\
\upsilon_{\alpha}(g,u)&=\upsilon_{\alpha}(cg).
\end{align*}
\end{remark}

Next, we collect some useful results. 
\begin{lemma}
\label{basicmu}
On a closed Riemannian manifold $(M,g(t))$, $t\in [0,T]$,  evolved by  (\ref{flow equation}), with $\D_{\alpha}\geq 0$. Let $\tau=T-t$, the following holds:\\
{\bf a.}  $\mathbb{W}_{\alpha}(g,\tau,h)$ is non-decreasing in time t (non-increasing in $\tau$).\\
{\bf b.} There exists a smooth minimizer $h_{\tau}$ for $\mathbb{W}_{\alpha}(g,\tau,.)$ which satisfies
\[
\tau(2\triangle{h_{\tau}}-|\nabla{h_{\tau}}|^2+\SS)+h_{\tau}-n=\mu_{\alpha}(g,\tau).
\]
In particular, $\mu_{\alpha}(g,\tau)$ is finite.\\
{\bf c.}   $\mu_{\alpha}(g,\tau)$ is non-decreasing in time t.\\
{\bf d.} $\lim_{\tau\rightarrow 0^{+}}\mu_{\alpha}(g,\tau)=0$.\\
\end{lemma}

\begin{proof}
Part {\bf a.} follows from \cite[Theorem 5.2]{GPT}.

Part {\bf b.} is deducted from the regularity theory for elliptic equations based on Sobolev spaces. The details can be found in \cite[Proposition 17.24]{chowetc3}. Replacing $\RR$ by $\SS$, the argument works exactly the same.

Part {\bf c.} is an immediate consequence of  the monotonicity formula (part {\bf a.}) and the existence of a minimizer realizing the $\mu_{\alpha}$ functional (part {\bf b.}).

The proof of part {\bf d.} is mostly identical to that of \cite[Prop 3.2]{stw03notes} (also \cite[Prop 17.19, 17.20]{chowetc3}), but it is subtle so we give a brief argument here.

Assume that the flow exists for $\tau\in [0, \overline{\tau}]$. The idea is to construct cut-off functions reflecting the local geometry which looks like Euclidean. Then it is shown that the limit of $\mathbb{W}_{\alpha}$ functional on these functions is 0 if a certain parameter approaches 0. Thus, by the monotonicity of $\mu_{\alpha}$ and L. Gross\rq{}s logarithmic-Sobolev inequality on an Euclidean space \cite{gross93log}, the result then follows.

The construction of cut-off functions follows \cite[Prop 3.2]{stw03notes}. Let $\tau_0=\overline{\tau}-\epsilon$ for small $\epsilon$. Using normal coordinates at a point p on $(M, g(\tau_0))$,  we define a cut-off function
\[ f_1=\begin{cases}
\frac{|x|^2}{4\epsilon} & \textrm{ if $d(x,p)=|x|<\rho$},\\
\frac{\rho^2}{4\epsilon} & \textrm {elsewhere},
\end{cases}
\]
where $\rho$ is a positive number smaller than the injectivity radius (which exists since $M$ is closed).
Then by the choice of our coordinate,

$$d\mu (\tau_0) =1+ O (|x|^2), |x|<<1 ,$$ and let  
$$e^{-C}\doteqdot\int_M (4\pi\epsilon)^{-n/2} e^{f_1},$$ then  $C\rightarrow 0$  as $ \epsilon \rightarrow 0$.\\

Let $f=f_1+C$ then,
\begin{align*}
u &\doteqdot (4\pi\epsilon)^{-n/2}e^{-f};\\
1 &=\int_M (4\pi\epsilon)^{-n/2}e^{-f} d\mu (\tau_0);\\
|\nabla f|^2 &=|\nabla f_1|^2=|\nabla \frac{|x|^2}{4\epsilon}|^2=\frac{|x|^2}{4\epsilon}, \textrm{ for $|x|<\rho$}.
\end{align*}
We solve f backward using  equation $$\ppt f=-\SS-\Delta f +|\nabla {f}|^2 +\frac{n}{2\tau}.$$  The solution clearly depends on the choice of $\epsilon$. Now using (\ref{defineW}), we calculate,
\begin{align*}
\mathbb{W} (g(\tau_0), \overline{\tau}-\tau_0,f(\tau_0))=&\int_{|x|<\rho} \Big(\epsilon(\SS+\frac{|x|^2}{4\epsilon^2})+\frac{|x|^2}{4\epsilon^2}+C-n\Big)u d\mu\\
&+\int_{d(x,p)\geq \rho} (\epsilon \SS+\frac{\rho^2}{4\epsilon}+C-n) u d\mu\\
&=\int_{|x|<\rho}(\frac{|x|^2}{2\epsilon}-n) u d\mu+ \epsilon \int_M \SS u d\mu\\
& + C\int_{M} u d\mu+ \int_{d(x,p)\geq \rho}(\frac{\rho^2}{4\epsilon}-n) u d\mu\\
&= I+ II+III+ IV.
\end{align*}
By a change of variable, as $\epsilon\rightarrow 0$, we have,
\begin{align*}
II+III &\rightarrow 0;\\
IV&= \int_{d(x,p)\geq \rho}(\frac{\rho^2}{4\epsilon}-n)e^{-\frac{\rho^2}{4\epsilon}-C} \rightarrow 0;\\
I&= e^{-C}\int_{|y|\leq \frac{\rho}{\sqrt{\epsilon}}} (\frac{|y|^2}{2}-n) (2\pi)^{-n/2} e^{-|y|^2/4}(1+O(\epsilon |y|^2) dy \\
&\rightarrow \int_{\mathbb{R}^n} (\frac{|y|^2}{2}-n) (2\pi)^{-n/2} e^{-|y|^2/4} dy=0.
\end{align*}

Thus, by part {\bf a.} and {\bf b}, $\mu_\alpha(g(t), \overline{\tau}-t)\leq 0$ for any $t\leq \overline{\tau}$. The proof that the limit is actually 0 when $\tau\rightarrow 0^{+}$ follows from a rather standard blow-up argument whose details can be found in
either  \cite[Prop 3.2]{stw03notes} or \cite[Prop 17.20]{chowetc3}.
\end{proof}
\section{Estimates on the Heat Kernel}
In this section, we obtain several estimates on the heat kernel using maximum principle and the monotone framework. Particularly, we derive a gradient estimate and an upper bound for positive solutions of the conjugate heat equation. Then we prove our main result.

\subsection{A Gradient Estimate} We first establish a space-only gradient estimate. 
Recall that, \[\Box^{\ast}=\frac{\partial}{\partial \tau}-\Delta+\SS.\]
\begin{lemma}
\label{gradconjS}
Assume there exist $k_{1}, k_{2}, k_{3}, k_{4}>0$ such that the followings hold on $N\times[0,T]$,
\begin{align*}
\Rc (g(t)) &\geq -k_{1}g(t),\\
\alpha &\geq -k_{2} g(t),\\
|\nabla{\SS}|^2 &\leq k_{3},\\
|\SS| & \leq k_{4}.
\end{align*}Let q be any positive solution to the conjugate heat equation on $M\times [0,T]$, i.e., $\Box^{\ast}q=0$ , and $\tau=T-t$. If $q<Q$ for some constant $Q$ then there exist $C_{1},C_{2}$ depending on $k_{1},k_{2},k_{3}, k_{4}$ and n, such that for $0<\tau\leq \min\{1,T\}$, we have
\begin{equation}
\tau \frac{|\nabla{q}|^2}{q^2}\leq (1+C_{1}\tau)(\ln{\frac{Q}{q}}+C_{2}\tau).
\end{equation}

\end{lemma}

\begin{proof}
We compute that
\begin{align*}
(-\ppt-\triangle)\frac{|\nabla{q}|^2}{q}=&\SS\frac{|\nabla{q}|^2}{q}+\frac{1}{q}(-\ppt-\triangle)|\nabla{q}|^2+2|\nabla{q}|^2\nabla{\frac{1}{q}}\nabla{\ln{q}}\\&-2\nabla|\nabla{q}|^2\nabla{\frac{1}{q}},\\
\frac{1}{q}(-\ppt-\triangle)|\nabla{q}|^2=&\frac{1}{q}\Big[-2(\alpha+\Rc)(\nabla{q},\nabla{q})-2\nabla{q}\nabla{(\SS q)}-2|\nabla^2{q}|^2 \Big],\\
2|\nabla{q}|^2\nabla{\frac{1}{q}}\nabla{\ln{q}}=&-2\frac{|\nabla{q}|^4}{q^3},\\
-2\nabla|\nabla{q}|^2\nabla{\frac{1}{q}}=&4\frac{\nabla^2{q}(\nabla{q},\nabla{q})}{q^2}.
\end{align*}
Thus
\begin{align*}
(-\ppt-\triangle)\frac{|\nabla{q}|^2}{q}=&\frac{-2}{q}|\nabla^2{q}-\frac{dq\otimes dq}{q}|^2+\SS\frac{|\nabla{q}|^2}{q}\\
&+\frac{-2(\alpha+\Rc)(\nabla{q},\nabla{q})-2\SS\nabla{q}\nabla{q}-2q\nabla{q}\nabla{\SS}}{q}\\
\leq &[2(k_{1}+k_{2})+nk_2]\frac{|\nabla{q}|^2}{q}+2|\nabla{q}||\nabla{\SS}|\\
\leq & [2k_{1}+(2+n)k_2+1]\frac{|\nabla{q}|^2}{q}+k_{3}q.
\end{align*}
Furthermore, we have
\begin{align*}
(-\ppt-\triangle)(q\ln{\frac{Q}{q}})&=-\SS q\ln{\frac{Q}{q}}+\SS q+\frac{|\nabla{q}|^2}{q}\\
& \geq \frac{|\nabla{q}|^2}{q}-nk_{2}q-k_{4}q\ln{\frac{Q}{q}}.
\end{align*}
Let $\Phi=a(\tau)\frac{|\nabla{q}|^2}{q}-b(\tau)q\ln{\frac{Q}{q}}-c(\tau)q,$ then
\begin{align*}
(-\ppt-\triangle)\Phi \leq &\frac{|\nabla{q}|^2}{q}\Big(a\rq{}(\tau)+a(\tau)(2k_{1}+(2+n)k_2+1)-b(\tau)\Big)\\
&+q\ln{\frac{Q}{q}}\Big(k_{4}b(\tau)-b\rq{}(\tau)\Big)\\
&+q\Big(k_{3} a(\tau)-c\rq{}(\tau)+nk_{2}b(\tau)+c(\tau)k_{4}\Big).
\end{align*}
We can now choose a, b and c appropriately such that $(-\partial_{t}-\triangle)\Phi\leq 0$. For example, take
\begin{align*}
a&=\frac{\tau}{1+(2k_{1}+(2+n)k_2+1)\tau},\\
b&=e^{k_{4}\tau},\\
c&=(e^{k_{5}k_{4}\tau}nk_{2}+k_{3})\tau,
\end{align*} for $k_{5}=1+\frac{k_3}{nk_2}$.
Then by maximum principle, noticing that $\Phi\leq 0$ at $\tau=0$, we arrive at
$$a\frac{|\nabla{q}|^2}{q}\leq b(\tau)q\ln{\frac{Q}{q}}+cq.$$
The result then follows from simple algebra.\\
\end{proof}

\subsection{$L_{\infty}$ Bound} Second, we shall derive an upper bound for positive conjugate heat solutions. Our main statement says that any normalized solution can not blow up too fast.
\begin{lemma}
\label{infconjS}
Let q be any normalized positive solution to the conjugate heat equation on $M\times [0,T]$, i.e., $\Box^{\ast}q=0$   with $\int q d\mu_{g(0)}=1$. Let $\tau=T-t$, then there exists a constant C depending on the geometry of ${g(t)}_{t\in[0,T]}$, such that
\begin{equation}
q(y,\tau)\leq \frac{C}{\tau^{n/2}}.
\end{equation}
\end{lemma}
\begin{proof}
The proof is modeled after \cite[Lemma 2.2]{ni06}  (also see \cite[Lemma 16.47]{chowetc2}).
As the solution and the flow is well defined in $M\times [0,T]$, there exists $y_{0}$, $\tau_{0}$ such that
\begin{equation}
\sup_{M\times [0,\min\{1,T\}]}\tau^{n/2}q(y,\tau)=\tau_{0}^{n/2}q(y_{0},\tau_{0}).
\end{equation}
In particular,
\begin{equation*}
\sup_{M\times [\tau_{0}/2,\tau_{0}]}q(y,\tau)\leq \frac{\tau_{0}^{n/2}}{\tau^{n/2}}q(y_{0},\tau_{0})\leq 2^{n/2}q(y_{0},\tau_{0}):=Q.
\end{equation*}
Applying Lemma \ref{gradconjS} to $q(y,\tau)$ on $M\times [\tau_0/2, \tau_0]$ we obtain,
\begin{equation*}
\frac{\tau_0}{2}\frac{|\nabla q|^2}{q^2}(y,\tau_{0}) \leq (1+C_{1}\frac{\tau_0}{2})(\log(\frac{Q}{q(y,\tau_0)})+C_{2}\frac{\tau_0}{2}).
\end{equation*}
Let $G(y,\tau_0):=\log(\frac{Q}{q(y,\tau_0)})+C_{2}\frac{\tau_0}{2}$, then the inequality above can be rewritten as
\begin{equation*}
|\nabla \sqrt{G}|^2=|\frac{1}{2}\frac{\nabla G}{G^{1/2}}|^2=|\frac{1}{2G^{1/2}}\frac{\nabla{q}}{q}|^2=\frac{1}{4G}\frac{|\nabla q|^2}{q^2}\leq \frac{1+ C_{1}\frac{\tau_0}{2}}{2\tau_0}.
\end{equation*}
Therefore, with $B_{\tau}(y,r)$ denoting the ball of radius r measured by $g(\tau)$ around the point y, we have
\begin{equation*}
\sup_{B_{\tau_0}(y_0, \sqrt{\frac{\tau_0}{1+ C_{1}\frac{\tau_0}{2}}})}\sqrt{G}(y,\tau_0)\leq \sqrt{G}(y_0,\tau_0)+\frac{1}{\sqrt{2}}=\sqrt{\frac{n}{2}\log{2}+C_{2}\frac{\tau_0}{2}}+\frac{1}{\sqrt{2}}.
\end{equation*}
Writing the above inequality in terms of $q(y,\tau_0)$ yields,
\begin{equation*}
q(y,\tau_0)\geq q(y_0,\tau_0)\exp \big \{-\frac{1}{2}-\frac{n}{2}\log{2}-\frac{2}{\sqrt{2}}\sqrt{\frac{n}{2}\log{2}+\frac{C_2}{2}}\big \}:=C_3 q(y_0,\tau_0).
\end{equation*}
Now we observe that there exists a constant $C_{4}$ depending on the geometry of $(M,g(\tau_0))$, such that
\begin{equation*}
\text{Vol}_{g(\tau_0)}\Big(B_{\tau_0}(y_0, \sqrt{\frac{\tau_0}{1+ C_{1}\frac{\tau_0}{2}}}\Big)\geq C_{4}\tau_0^{n/2}.
\end{equation*}
Therefore we have,
\begin{equation*}
q(y_0,\tau_0)\leq \frac{1}{C_3C_4\tau_0^{n/2}}\int_{M}q(y,\tau_0)d\mu_{\tau_0}(y):=\frac{C_5}{\tau_0^{n/2}}\int_{M}q(y,\tau_0)d\mu_{\tau_0}(y).
\end{equation*}
By our choice of $y_0, \tau_0$ and the fact that $\int_{M}q(y,\tau)d\mu_{\tau}(y)$ remains constant along the flow, the statement follows.
\end{proof}

\begin{remark}
It is interesting to note that in \cite{z12bounds}, a Harnack inequality is used to obtain an off-diagonal bound,  while here the argument goes the opposite direction.\end{remark}
\subsection{Proofs of Main Results} Finally, we are ready to finish our proof of the main theorem.

\begin{lemma}
\label{integralboundabove} Let $H(x,t;y,T)=(4\pi\tau)^{-n/2}e^{-h}$ be a heat kernel and $\Phi$ be any positive solution to the heat equation. Then we have
\[\int_{M}hH\Phi d\mu\leq \frac{n}{2}\Phi(y,T), ~\text{i.e, }~\int_{M}(h-\frac{n}{2})H\Phi d\mu \leq 0.\]
\end{lemma}

\begin{proof}
By Lemma \ref{compareRDandconjheat} we have
\begin{align*}
\limsup_{\tau\rightarrow 0}\int_M hH\Phi d\mu \leq \limsup_{\tau\rightarrow 0}\int_M \ell (x,\tau)H\Phi d\mu (x)\\
\leq \limsup_{\tau\rightarrow 0}\int_{M} \frac{d_{T}^2(x,y)}{4\tau}H\Phi d\mu(x).
\end{align*}
Using Theorem \ref{asymptoticconj}, it follows that, 
\begin{equation*}
\lim_{\tau\rightarrow 0}\int_{M} \frac{d_{T}^2(x,y)}{4\tau}H\Phi d\mu (x)=\lim_{\tau\rightarrow 0}\int_{M} \frac{d_{T}^2(x,y)}{4\tau}\frac{e^{-\frac{d_{T}^2(x,y)}{4\tau}}}{(4\pi\tau)^{n/2}}\Phi d\mu(x).
\end{equation*}
Either by differentiating twice under the integral sign or using these following identities on Euclidean spaces
\[\int_{-\infty}^{\infty}e^{-a\textbf{x}^2}d\textbf{x} =\sqrt{\frac{\pi}{a}} \text{ and } \int_{-\infty}^{\infty}\textbf{x}^{2}e^{-a\textbf{x}^2}d\textbf{x} =\frac{1}{2a}\sqrt{\frac{\pi}{a}},
\]
we obtain that 
\begin{equation*}
\int_{\RR^{n}}|x|^{2}e^{-a|x|^2}dx=n(\int_{-\infty}^{\infty}\textbf{x}^{2}e^{-a\textbf{x}^2}d\textbf{x})\Big(\int_{-\infty}^{\infty}e^{-a\textbf{x}^2}d\textbf{x}\Big)^{n-1}=\frac{n}{2a}(\frac{\pi}{a})^{n/2}.
\end{equation*}
Therefore,
\begin{equation*}
\lim_{\tau\rightarrow 0}\frac{d_{T}^2(x,y)}{4\tau}\frac{e^{-\frac{d_{T}^2(x,y)}{4\tau}}}{(4\pi\tau)^{n/2}}=\frac{n}{2}\delta_{y}(x),
\end{equation*}
hence
\begin{equation*}
\lim_{\tau\rightarrow 0}\int  \frac{d_{T}^2(x,y)}{4\tau}\frac{e^{-\frac{d_{T}^2(x,y)}{4\tau}}}{(4\pi\tau)^{n/2}}\Phi d\mu_{N}(x)=\frac{n}{2}\Phi(y,T).
\end{equation*}
Thus the result follows.
\end{proof}
 The following result implies that  the equality actually holds.

\begin{proposition}
\label{limit0}
Let $H(x,t;y,T)=(4\pi\tau)^{-n/2}e^{-h}$ be a heat kernel and $\Phi$ be any positive solution to the heat equation. Then for
\begin{align*}
v&=\Big[(T-t)(2\triangle{h}-|\nabla{h}|^2+S)+h-n\Big]H,\\
\rho_{\Phi}(t)&=\int_M v\Phi d\mu,
\end{align*}
 we have
\[\lim_{t\rightarrow T}\rho_{\Phi}(t)=0.\]
\end{proposition}

\begin{proof}

Integrating by parts yields that 
\begin{align*}
\rho_{\Phi}(t) &=\int_{M}\Big[\tau(2\triangle{h}-|\nabla h|^2+S)+h-n\Big]H\Phi d\mu\\
&=-\int_{M}2\tau\nabla{h}\nabla({H\Phi})d\mu-\int_{M}\tau|\nabla h|^2 H\Phi d\mu+\int_{M}(\tau \SS+h-n)H\Phi d\mu\\
&= \int_{M}\tau|\nabla h|^2 H\Phi d\mu-2\tau\int_{M} \nabla{\Phi}\nabla{h}H d\mu+\int_{M}(\tau \SS+h-n)H\Phi d\mu\\
&= \int_{M}\tau|\nabla h|^2 H\Phi d\mu-2\tau\int_{M}H\triangle{\Phi}d\mu+\int_{M}(\tau \SS+h-n)H\Phi d\mu\\
&= \int_{M}\tau|\nabla h|^2 H\Phi d\mu+\int_{M}hH\Phi d\mu-2\tau\int_{M}H\triangle{\Phi}d\mu+\int_{M}(\tau\SS-n)H\Phi d\mu.
\end{align*}

Notice that, except the first two terms, the rest approaches $-n\Phi(y,T)$ as $\tau\rightarrow 0$.
For the first term, using Lemmas \ref{infconjS} and \ref{gradconjS} for any space-($\tau$) time point on $M\times[\frac{\tau}{2},\tau]$ we arrive at
\begin{align*}
\tau\int_{M}|\nabla h|^2H\Phi d\mu&\leq (2+C_{1}\tau)\int_{M}(\ln{(\frac{C_{3}}{H\tau^{n/2}})}+C_{2}\tau)H\Phi d\mu\\
& \leq (2+C_{1}\tau)\int_{M}(\ln{C_{3}}+h+\frac{n}{2}\ln(4\pi)+C_{2}\tau)H\Phi d\mu,
\end{align*}
with $C_{1},C_{2}$ defined as in Lemma \ref{gradconjS}, while $C_{3}$ is a constant depending on the geometry of $g(t)$, $\frac{\tau}{2}\leq T-t\leq\tau$. As $\tau\rightarrow 0$, $\ln C_{3}+\frac{n}{2}\ln(4\pi) $ is bounded from above by another constant $C_{4}$ also depending on the geometry of $g(t)$, $t\in[0,T]$.
Consequently, by Lemma \ref{integralboundabove}, which claims the finiteness of $\int_M hH\Phi d\mu$,
\begin{align*}
\lim_{\tau\rightarrow 0}(\int_{M}\tau|\nabla h|^2d\mu+\int_{M}hH\Phi d\mu)&\leq 3\int_{M}hH\Phi d\mu+2\ln{C_{4}}\Phi(x,T)\\
&\leq (\frac{3n}{2}+2\ln{C_{4}})\Phi(x,T).
\end{align*}
Thus we have 
\begin{equation*}
\lim_{t\rightarrow T}\rho_{\Phi}(t)\leq C_{5}\Phi(x,T).
\end{equation*}

Since $\Phi$ is a positive solution satisfying $\partial_{t}\Phi=\triangle{\Phi}$, applying Theorem \ref{conjHar} yields that,
\begin{equation}
\label{evolintegral}
\partial_{t}\rho_{\Phi}(t)=\partial_{t}\int v\Phi d\mu=\int (\Box{\Phi}v-\Phi\Box^{\ast}v)d\mu \geq 0.
\end{equation}
The above implies that there exists $\beta$, such that
\begin{equation*}
\lim_{t\rightarrow T}\rho_{\Phi}(t)=\beta.
\end{equation*}
Hence \[\lim_{\tau\rightarrow 0}(\rho_{\Phi}(T-\tau)-\rho_{\Phi}(T-\frac{\tau}{2}))=0.\]
By the above equation (\ref{evolintegral}), Theorem \ref{conjHar}, and the mean-value theorem, there exists a sequence $\tau_{i}\rightarrow 0$, such that
\begin{equation*}
\lim_{\tau_{i}\rightarrow 0}\tau_{i}^2\int_{M}\Big(|\alpha+\text{Hess}h-\frac{g}{2\tau_i}|^2+\frac{1}{2}\D_{\alpha}(\nabla{h})\Big)H\Phi d\mu=0.
\end{equation*}
Now standard inequalities yield that,
\begin{align*}
&\Big[\int_{M}\tau_{i}(\SS+\triangle{h}-\frac{n}{2\tau_{i}})H\Phi d\mu\Big]^2 \\
\leq& \Big[\int_{M}\tau_{i}^2(S+\triangle{h}-\frac{n}{2\tau_{i}})^2H\Phi d\mu\Big] \Big[\int_{M}H\Phi d\mu\Big]\\
 \leq& n \Big[\int_{M}\tau_{i}^2|\mathcal{S}+\text{Hess}h-\frac{g}{2\tau_i}|^2H\Phi d\mu\Big] \Big[\int_m H\Phi d\mu\Big].
\end{align*}
Since $$\lim_{\tau_{i}\rightarrow 0}\int_{M}H\Phi d\mu=\Phi(y,T)<\infty,$$ and $\frac{1}{2}\D_{\alpha}(\nabla{h})\geq 0$, we derive that
\begin{equation*}
\lim_{\tau_{i}\rightarrow 0}\int_{M}\tau_{i}(S+\triangle{h}-\frac{n}{2\tau_{i}})H\Phi d\mu=0.
\end{equation*}
Therefore, by Lemma \ref {integralboundabove},
\begin{align*}
\lim_{t\rightarrow T}\rho_{\Phi}(t)&=\lim_{\tau_{i}\rightarrow 0}\int_M \Big[\tau_{i}(2\triangle{h}-|\nabla{h}|^2+S)+h-n\Big]H\Phi d\mu\\
&=\lim_{\tau_{i}\rightarrow 0}\int_{M}\Big[\tau_{i}(\triangle{h}-|\nabla{h}|^2)+h-\frac{n}{2}\Big]H\Phi d\mu\\
&=\lim_{\tau_{i}\rightarrow 0}\Big[\int_{M}-\tau_{i}H\triangle{\Phi}d\mu+\int_{M}(h-\frac{n}{2})H\Phi d\mu\Big]\\
&=\lim_{\tau_{i}\rightarrow 0}\int_{M}(h-\frac{n}{2})H\Phi d\mu\leq 0.
\end{align*}
So $\beta\leq 0$. To show that equality holds, we proceed by contradiction. Without loss of generality, we may assume $\Phi(y,T)=1$. Let $H\Phi=(4\pi\tau)^{-n/2}e^{-\tilde{h}}$ (that is, $\tilde{h}=h-\ln{\Phi}$), then integrating by parts yields,
\begin{equation}
\rho_{\Phi}(t)=\mathbb{W}_{\alpha}(g,\tau,\tilde{h})+\int_{M}\Big(\tau(\frac{|\nabla{\Phi}|^2}{\Phi})-\Phi\ln{\Phi}\Big)H d\mu.
\end{equation}
By the choice of $\Phi$ the last term converges to 0 as $\tau\rightarrow 0$. So if $\lim_{t\rightarrow T}\rho_{\Phi}(t)=\beta<0$ then $\lim_{\tau\rightarrow 0}\mu_{\alpha}(g,\tau)<0$ and, thus, contradicts Lemma \ref{basicmu}. Therefore $\beta=0$.
\end{proof}
Now Theorem \ref{mainth} follows immediately.
\begin{proof}(Theorem \ref{mainth})
 Recall from inequality (\ref{evolintegral})
$$\partial_{t}\int_M v\Phi d\mu=\int_{M}(\Box{\Phi}v-\Phi\Box^{\ast}v)d\mu\geq 0.$$
By Proposition \ref{limit0}, $$\lim_{t\rightarrow T}\int_{M}v\Phi d\mu=0.$$ Since $\Phi$ is arbitrary, $v\leq 0$.
\end{proof}

\def\cprime{$'$}
\bibliographystyle{plain}
\bibliography{bio}

\end{document}